\documentclass[12pt, reqno]{amsart}
\usepackage{amsmath, amsthm, amscd, amsfonts, amssymb, graphicx, color}
\usepackage[bookmarksnumbered, colorlinks, plainpages]{hyperref}
\usepackage{cite}
\theoremstyle{definition}
\newtheorem{defn}{Definition}[section]
\newtheorem{example}{Example}[section]
\newtheorem{thm}{Theorem}[section]

\newtheorem{prop}{Proposition}[section]

\newcommand{\be}{\begin{equation}}
	\newcommand{\ee}{\end{equation}}
\newcommand{\beas}{\begin{eqnarray*}}
	\newcommand{\eeas}{\end{eqnarray*}}
\newcommand{\bea}{\begin{eqnarray}}
	\newcommand{\eea}{\end{eqnarray}}
\numberwithin{equation}{section}

\begin{document}
	\setcounter{page}{1}
	
\title[Complex symmetric weighted composition operators on the space $\mathcal{H}^2_2(\mathbb{D})$]{Complex symmetric weighted composition operators on the space $\mathcal{H}^2_2(\mathbb{D})$}
	
	\author[M. B. Ahamed and T. Rahman]{Molla Basir Ahamed$^*$ and Taimur Rahman}
	
	\address{Department of Mathematics, Jadavpur University, Kolkata-700032, West Bengal, India.}
	\email{mbahamed.math@jadavpuruniversity.in}
	
		\address{%
		Department of Mathematics, Jadavpur University, Kolkata-700032, West Bengal, India.}
	\email{taimurr.math.rs@jadavpuruniversity.in}
	
	
	\date{Received: xxxxx; 
		\newline\indent $^{1}$ Corresponding author
		\newline\indent File:- Paper-AT-5-2023}
	
\begin{abstract} 
In this paper, we introduce a new norm for $\mathcal{S}^2(\mathbb{D})$, encompassing functions whose first and second derivatives belong to both the Hardy space $\mathcal{H}^2(\mathbb{D})$ and the classical Bergman space $\mathcal{A}^2(\mathbb{D})$. Moreover, we present some basic properties of the space $\mathcal{H}^2_2(\mathbb{D})$ and subsequently establish conditions for symbols $\phi$ and $\Psi$ to provide $W_{\Psi, \phi}$ complex symmetric, employing a unique conjugation $\mathcal{J}$.
\end{abstract} \maketitle
	
\section{{Introduction}}
Let $\mathbb{D}$ represent the open unit disk in the complex plane $\mathbb{C}$, and let $H(\mathbb{D})$ denote the space comprising all holomorphic functions in this disk. If $f\in H(\mathbb{D})$ with power series representation $f(z)=\sum_{n=0}^{\infty}a_nz^n$, then the Hardy space is defined by 
\begin{align*}
	\mathcal{H}^2(\mathbb{D})=\bigg\{f\in\ H(\mathbb{D}):||f||^2_{\mathcal{H}^2}=\sum_{n=0}^{\infty}|a_n|^2<\infty\bigg\}
\end{align*}
and the classical Bergman space is defined by 
\begin{align*}
	\mathcal{A}^2(\mathbb{D})=\bigg\{f\in\ H(\mathbb{D}):||f||^2_{\mathcal{A}^2}=\sum_{n=0}^{\infty}\frac{|a_n|^2}{n+1}<\infty\bigg\}.
\end{align*}
The composition operators have been the subject of extensive research on these spaces in \cite{Cowen-MacCluer-CRCP-1995, Shapiro-SVNY-1993}. For other aspects of composition operators, we refer to the articles \cite{Higdon-JFA-2005,Liu-Ponnusamy-Xie-LMA-2023,Lo-Loh-JMAA-2023,Manhas-Zhao-JMAA-2012,Martin-Vukotic-JFA-2006,Martin-vukotic-PAMS-2006,McDonald-PAMS-2003,Mengestie-Seyoum-QM-2021,Mengestie-Seyoum-CMFT-2021,Noor-Severiano-PAMS-2020,Ohno-BAMS-2006,Ueki-PAMS-2007,Yao-JMAA-2017,Wulan-Zheng-Zhu-PAMS-2009} and references therein. There is the space $\mathcal{S}^2(\mathbb{D})$ (see \cite{Allen-Heller-Pons-ACM-2015, Heller-2010}) which encompasses holomorphic functions on $\mathbb{D}$, with the condition that their first derivative lies in the Hardy space $\mathcal{H}^2(\mathbb{D})$ and is defined as follows:
\begin{align*}
	\mathcal{S}^2(\mathbb{D})=\{f\in\ H(\mathbb{D}): ||f||^2_{\mathcal{S}^2}=|f(0)|^2+||f^{\prime}||^2_{\mathcal{H}^2}<\infty\}.
\end{align*}
In \cite{Gu-Luo-CVEE-2018}, Gu and Luo defined an equivalent norm on $\mathcal{S}^2(\mathbb{D})$, and denote the space by $\mathcal{S}^2_1(\mathbb{D})$.
\begin{align*}
	\mathcal{S}^2_1(\mathbb{D})&=\bigg\{f\in\ H(\mathbb{D}): ||f||^2_{\mathcal{S}^2_1}=||f||^2_{\mathcal{H}^2}+\frac{3}{2}||f^{\prime}||^2_{\mathcal{H}^2}+\frac{1}{2}||f^{\prime}||^2_{\mathcal{A}^2}<\infty\bigg\}\\&=\bigg\{f\in\ H(\mathbb{D}): ||f||^2_{\mathcal{S}^2_1}=\sum_{n\ge 0}\frac{(n+1)(n+2)}{2}|a_n|^2<\infty\bigg\}.
\end{align*}
This space is called the derivative Hardy space (see \cite{Gu-Luo-CVEE-2018}). Motivated by the definition of the derivative Hardy space, Gupta and Malhotra \cite{Gupta-Malhotra-CVEE-2020} introduced an equivalent norm for the space $\mathcal{S}^2(\mathbb{D})$ and referred to it as $\mathcal{H}^2_1(\mathbb{D})$. This space is defined as follows:
\begin{align*}
	\mathcal{H}^2_1(\mathbb{D})&=\bigg\{f\in\ H(\mathbb{D}): ||f||^2_{\mathcal{H}^2_1}=7||f||^2_{\mathcal{H}^2}+||f^{\prime}||^2_{\mathcal{H}^2}+||f||^2_{\mathcal{A}^2}+5||f^{\prime}||^2_{\mathcal{A}^2}<\infty\bigg\}\\&=\bigg\{f\in\ H(\mathbb{D}): ||f||^2_{\mathcal{H}^2_1}=7\sum_{n=0}^{\infty}|a_n|^2+\sum_{n=0}^{\infty}n^2|a_n|^2+\sum_{n=0}^{\infty}\frac{|a_n|^2}{n+1}\\&\quad\quad\quad\quad\quad\quad\quad\quad\quad\quad\quad+5\sum_{n=0}^{\infty}n|a_n|^2<\infty\bigg\}\\&=\bigg\{f\in\ H(\mathbb{D}): ||f||^2_{\mathcal{H}^2_1}=\sum_{n\ge 0}\frac{(n+2)^3}{n+1}|a_n|^2<\infty\bigg\}.	
\end{align*}
A reproducing kernel Hilbert space, denoted as $\mathcal{H}$, is characterized by its inner product $\langle ., .\rangle_{\mathcal{H}}$, which has the property that for every $w\in\mathbb{D}$, there is a function $K_w\in\mathcal{H}$, called the point evaluation kernel, such that $f(w) = \langle f, K_w\rangle_{\mathcal{H}}$ for all $f\in\mathcal{H}$. The Hardy space, Bergman space, $\mathcal{S}^2(\mathbb{D})$ space, and $\mathcal{H}^2_1(\mathbb{D})$ space all belong to the class of reproducing kernel Hilbert spaces, and their respective kernels are defined as follows:
\begin{align*}
	\begin{cases}
		&K_w(z)=\dfrac{1}{(1-\overline{w}z)}\; \mbox{on}\; \mathcal{H}^2(\mathbb{D});\\&K_w(z)=\dfrac{1}{(1-\overline{w}z)^2}\; \mbox{on}\; \mathcal{A}^2(\mathbb{D});\\&K_w(z)=\displaystyle\sum_{n=0}^{\infty}\frac{(\overline{w}z)^n}{1+n^2},\; \mbox{on}\;\mathcal{S}^2(\mathbb{D});\\&K_w(z)=\displaystyle\sum_{n=0}^{\infty}\frac{n+1}{(n+2)^3}(\overline{w}z)^n\; \mbox{on}\;\mathcal{H}^2_1(\mathbb{D}).
	\end{cases}
\end{align*}

With the definition of the derivative Hardy space as our motivation, our aim is to establish an equivalent norm for the space $\mathcal{S}^2(\mathbb{D})$, which we will refer to as $\mathcal{H}^2_2(\mathbb{D})$. The new space $\mathcal{H}^2_2(\mathbb{D})$ is defined as follows:
\begin{align*}
	\mathcal{H}^2_2(\mathbb{D})&=\bigg\{f\in\ H(\mathbb{D}): ||f||^2_{\mathcal{H}^2_2}=31||f||^2_{\mathcal{H}^2}+41||f^{\prime}||^2_{\mathcal{H}^2}+||f^{\prime\prime}||^2_{\mathcal{H}^2}+||f||^2_{\mathcal{A}^2}\\&\quad\quad\quad\quad\quad\quad\quad\quad\quad\quad\quad+49||f^{\prime}||^2_{\mathcal{A}^2}+11||f^{\prime\prime}||^2_{\mathcal{A}^2}<\infty\bigg\}\\&=\bigg\{f\in\ H(\mathbb{D}): ||f||^2_{\mathcal{H}^2_2}=31\sum_{n=0}^{\infty}|a|^2+41\sum_{n=0}^{\infty}n^2|a|^2+\sum_{n=0}^{\infty}n^2(n-1)^2|a|^2\\&\quad\quad\quad\quad\quad\quad\quad+\sum_{n=0}^{\infty}\frac{|a|^2}{n+1}+49\sum_{n=0}^{\infty}n|a|^2+11\sum_{n=0}^{\infty}n^2(n-1)|a|^2<\infty\bigg\}\\&=\bigg\{f\in\ H(\mathbb{D}): ||f||^2_{\mathcal{H}^2_1}=\sum_{n\ge 0}\frac{(n+2)^5}{n+1}|a_n|^2<\infty\bigg\}.		
\end{align*}
For $f, g\in\mathcal{H}^2_2(\mathbb{D})$, with power series representation $f(z)=\sum_{n=0}^{\infty}a_nz^n$ and $g(z)=\sum_{n=0}^{\infty}b_nz^n$, the inner product on $\mathcal{H}^2_2(\mathbb{D})$ is defined by
\begin{align*}
	\langle f,g \rangle_{\mathcal{H}^2_2}=&31\langle f,g \rangle_{\mathcal{H}^2}+41\langle f^{\prime},g^{\prime} \rangle_{\mathcal{H}^2}+\langle f^{\prime\prime},g^{\prime\prime} \rangle_{\mathcal{H}^2}+\langle f,g \rangle_{\mathcal{A}^2}+49\langle f^{\prime},g^{\prime} \rangle_{\mathcal{A}^2}+11\langle f^{\prime\prime},g^{\prime\prime} \rangle_{\mathcal{A}^2}\\=&31\sum_{n=0}^{\infty}a_n\overline{b_n}+41\sum_{n=0}^{\infty}n^2a_n\overline{b_n}+\sum_{n=0}^{\infty}n^2(n-1)^2a_n\overline{b_n}+\sum_{n=0}^{\infty}\frac{a_n\overline{b_n}}{n+1}+49\sum_{n=0}^{\infty}na_n\overline{b_n}\\&+11\sum_{n=0}^{\infty}n^2(n-1)a_n\overline{b_n}\\=&\sum_{n\ge 0}\frac{(n+2)^5}{n+1}a_n\overline{b_n}.
\end{align*}
For $w\in\mathbb{D}$, we define 
\begin{align*}
	K_w(z):=\sum_{n=0}^{\infty}\frac{n+1}{(n+2)^5}(\overline{w}z)^n \;\; \mbox{for}\;\; z\in\mathbb{D}.
\end{align*}
An easy computation shown that 
\begin{align*}
	\langle f(z), K_w(z) \rangle_{\mathcal{H}^2_2}=\sum_{n=0}^{\infty}\frac{(n+2)^5}{n+1}a_n\overline{\bigg(\frac{(n+1)}{(n+2)^5}\overline{w}^n\bigg)}=\sum_{n=0}^{\infty}a_nw^n=f(w), z\in\mathbb{D}.
\end{align*}
Hence, $K_w$ is the reproducing kernel for the space $\mathcal{H}^2_2(\mathbb{D})$.\vspace{1.2mm}

In this paper, we initiate by introducing some fundamental properties of the space $\mathcal{H}^2_2(\mathbb{D})$. We then proceed to investigate the criteria under which the composition operator $C_{\phi}:\mathcal{H}^2_2(\mathbb{D})\to\mathcal{H}^2_2(\mathbb{D})$, symbolized by $\phi$ and defined as $C_{\phi}f=f\circ\phi$ where $f$ is holomorphic on $\mathbb{D}$ and $\phi$ is a holomorphic map from $\mathbb{D}$ to itself, qualifies as a Hilbert-Schmidt operator.\vspace{1.2mm}

The study of complex symmetric operators on Hilbert spaces was initiated by Gartia and Putinar \cite{Garcia-Putinar-TAMS-2006,Garcia-Putinar-TAMS-2007}. Since then, this area has attracted significant research attention in complex analysis and operator theory. Complex symmetric operators include various instances, such as normal operators, Volterra integration operators, Hankel operators, Toeplitz operators, and more. Let $\Psi:\mathbb{D}\to\mathbb{C}$ and $\phi:\mathbb{D}\to\mathbb{D}$ be holomorphic functions. Then the weighted composition operator $W_{\Psi, \phi}$ is defined as 
\begin{align*}
	W_{\Psi, \phi}f=\Psi\cdot f\circ\phi.
\end{align*}
Several authors in \cite{Bourdon-Narayan-JMAA-2010,Garcia-Hammond-OPAP-2013,Jun-Kim-Ko-Lee-JFA-2014} studied normal and complex symmetric weighted composition operators with the special conjugation on the Hardy space $\mathcal{H}^2(\mathbb{D})$. Khoi and  Lim \cite{Khoi-Lim-JMAA-2018} studied the complex symmetry of weighted composition operators on the Hilbert space of analytic functions  $\mathcal{H}_{\gamma}(\mathbb{D})$ over $\mathbb{D}$ with reproducing kernels $K^{(\gamma)}_w=(1-\overline{w}z)^{-\gamma}$, where $\gamma\in\mathbb{N}$. Hai and Khoi \cite{Hai-Khoi-JMAA-2016,Hai-Khoi-CVEE-2018} established the conditions under which a weighted composition operator exhibits complex symmetry in the Fock space (see \cite{Hai-Khoi-JMAA-2016,Hai-Khoi-CVEE-2018}).\vspace{1.2mm} 

The outline of the paper is as follows. In Section 2, we characterize the space $\mathcal{H}^2_2(\mathbb{D})$ by exploring the basis. In the same spirit, In Section 3, we study multiplication operator on the space $\mathcal{H}^2_2(\mathbb{D})$. In the Section 4, of this paper, we explore the properties of $\phi$ and $\Psi$ concerning the weighted composition operator $W_{\phi, \Psi}$ in the space $\mathcal{H}^2_2(\mathbb{D})$, with the specific aim of confirming its complex symmetry. The proof of the main results are discussed in each sections.

\section{Characterization of the space $\mathcal{H}^2_2(\mathbb{D})$}

In this section, we will provide an account of some essential properties of the space $\mathcal{H}^2_2(\mathbb{D})$ that will be drawn upon in later sections. \vspace{1.2mm}
\begin{prop}\label{pro-2.1}
	The space $\mathcal{H}^2_2(\mathbb{D})$ satisfies the following:
	\begin{enumerate}
		\item [(a)] $\mathcal{H}^2_2(\mathbb{D})$ is a normed space. 
		\item[(b)] polynomials are contained in $\mathcal{H}^2_2(\mathbb{D})$.
		\item[(c)]$\mathcal{H}^2_2(\mathbb{D})$ is a proper subset of $\mathcal{H}^2(\mathbb{D})$,
	\end{enumerate}
	The following example serves as proof of the proper inclusion of part (c) in Proposition \ref{pro-2.1}.
\end{prop}
\begin{example}
	Let $f(z)=\sum_{n=0}^{\infty}a_nz^n=\sum_{n=0}^{\infty}\sqrt{(n+1)/(n+2)^5}z^n$. Then 
	\begin{align*}
		||f||_{\mathcal{H}^2}=\sum_{n=0}^{\infty}|a_n|^2=\sum_{n=0}^{\infty}\frac{n+1}{(n+2)^5}=\zeta(4)-\zeta(5)\approxeq0.04539<\infty.
	\end{align*}
	Since
	\begin{align*}
		\sum_{n=0}^{\infty}\frac{(n+2)^5}{n+1}|a_n|^2=\sum_{n=0}^{\infty}\frac{(n+2)^5}{n+1}\frac{n+1}{(n+2)^5}=\sum_{n=0}^{\infty}1=\infty,
	\end{align*}
	so $f\notin\mathcal{H}^2_2(\mathbb{D})$.
	Hence, $f$ belongs to $\mathcal{H}^2(\mathbb{D})$ but it does not belong to $\mathcal{H}^2_2(\mathbb{D})$.   
\end{example}
To gain a comprehensive understanding of any space, it is essential to be familiar with its basis. For instance, when attempting to demonstrate that an operator defined on a given space is a Hilbert-Schmidt operator, it is necessary to have an orthonormal basis for that space. Since space $\mathcal{H}^2_2(\mathbb{D})$ has been recently constructed, understanding its basis is crucial for a thorough comprehension of this new space. The following theorem provides an orthonormal basis for the space $\mathcal{H}^2_2(\mathbb{D})$.
\begin{thm}\label{thm-2.1}
	Suppose $e_n(z)=\sqrt{(n+1)/(n+2)^5}z^n$, $z\in\mathbb{D}$ and $n$ is any non-negative integer. Then the set $ \{e_n(z):n\in\mathbb{N}\cup\{0\}\}$ forms an orthonormal basis for $\mathcal{H}^2_2(\mathbb{D})$.
\end{thm}
\begin{proof}[\bf Proof of Theorem \ref{thm-2.1}]
	For any non-negative integers $m$ and $n$, we have
	\begin{align}\label{Eq-2.1}
		\langle z^m, z^n\rangle_{\mathcal{H}^2_2}=31&\langle z^m, z^n\rangle_{\mathcal{H}^2}+41\langle mz^{m-1}, nz^{n-1}\rangle_{\mathcal{H}^2}+\langle m(m-1)z^{m-2}, n(n-1)z^{n-2}\rangle_{\mathcal{H}^2}\\+&\langle z^m, z^n\rangle_{\mathcal{A}^2}+49\langle mz^{m-1}, nz^{n-1}\rangle_{\mathcal{A}^2}+11\langle m(m-1)z^{m-2}, n(n-1)z^{n-2}\rangle_{\mathcal{A}^2}\nonumber.
	\end{align}
	For $n=m$, we obtain that 
	\begin{align*}
		\langle z^m, z^n\rangle_{\mathcal{H}^2_2}=&31\cdot1+41n^2+n^2(n-1)^2+\frac{1}{(n+1)}+49n+11n^2(n-1)\\=&\frac{n^5+10n^4
			+40n^3+80n^2+80n+32}{(n+1)}\\=&\frac{(n+2)^5}{(n+1)}.
	\end{align*}
	Therefore, it is easy to see that 
	\begin{align*}
		\langle z^m, z^n\rangle_{\mathcal{H}^2_2}=\begin{cases}
			\frac{(n+2)^5}{(n+1)}\;\;\;\;\;\mbox{if}\; m=n\\\;\;
			0\;\;\;\;\;\;\;\;\;\; \mbox{if}\;m\neq n.
		\end{cases}	
	\end{align*}
	Hence, $ \{e_n(z):n\in\mathbb{N}\cup\{0\}\}= \{\sqrt{(n+1)/(n+2)^5}z^n:n\in\mathbb{N}\cup\{0\}\}$ forms an orthonormal set in $\mathcal{H}^2_2(\mathbb{D})$.\vspace{1.2mm}
	
	Let $f\in\mathcal{H}^2_2(\mathbb{D})$ be such that $f$ is orthogonal to $e_n(z)$ for all $n\ge0$ and let the power series expansion of $f$ be given by $f(z)=\sum_{m=0}^{\infty}a_mz^m$. Then for $n\ge0$, $\langle f(z),e_n(z)\rangle=\langle\sum_{m=0}^{\infty}a_mz^m, e_n(z)\rangle=0$. This gives $a_m=0$ for all $m\ge0$. Thus, $f\equiv0$ and hence the set  $ \{e_n(z):n\in\mathbb{N}\cup\{0\}\}$ forms an orthonormal basis for $\mathcal{H}^2_2(\mathbb{D})$.  
\end{proof}
\section{Multiplication operator on the space  $\mathcal{H}^2_2(\mathbb{D})$}
Let $\mathcal{B}$ be a Banach space consisting of holomorphic functions defined on the open unit disk $\mathbb{D}$.Then for any function $f\in\mathcal{B}$, the multiplication operator $M_f$ is defined as follows:
\begin{align*}
	M_fg=fg
\end{align*}
for every holomorphic function $g$ defined on $\mathbb{D}$. Multiplication operators have a key role to play in the analysis of weighted composition operators in holomorphic function spaces. The aim is to find a connection between the properties of the multiplication operator $M_f$ and those of its corresponding symbol $f$. In this section, we will examine the circumstances that determine the boundedness of the multiplication operator $M_f$ on $\mathcal{H}^2_2(\mathbb{D})$.\vspace{1.2mm}

In a recent study, Gu and Luo (see \cite[Proposition 2.1]{Gu-Luo-CVEE-2018}) demonstrated that for functions belonging to $\mathcal{S}^2_1(\mathbb{D})$, it holds that $||f||{\infty}\le\sqrt{2}||f||{\mathcal{S}^2_1}$ where the constant $\sqrt{2}$ is sharp. Subsequently, Gupta and Malhotra (in Theorem-3.1, as referenced in \cite{Gupta-Malhotra-CVEE-2020}) extended this result by establishing that for functions in $\mathcal{H}^2_1(\mathbb{D})$, the norm $||f||{\infty}\le\sqrt{\pi^2/6-\zeta(3)}||f||{\mathcal{H}^2_1}$, and it was confirmed that the constant $\sqrt{\pi^2/6-\zeta(3)}$ is sharp.\vspace{1.2mm}

In the following result, we will introduce an analogous counterpart to Proposition 2.1 mentioned in \cite{Gu-Luo-CVEE-2018} and Theorem 3.1 discussed in \cite{Gupta-Malhotra-CVEE-2020}.
\begin{thm}\label{thm-3.1}
	Suppose $f\in\mathcal{H}^2_2(\mathbb{D})$. Then $||f||_{\infty}\le\sqrt{\zeta(4)-\zeta(5)}||f||_{\mathcal{H}^2_2}$. Furthermore, $\sqrt{\zeta(4)-\zeta(5)}$ is the sharp constant($\zeta$ is the Riemann zeta function and $\zeta(4)\approxeq1.08232$ and $\zeta(5)\approxeq1.03693$).
\end{thm}
\begin{proof}[\bf Proof of Theorem \ref{thm-3.1}]
	Suppose $f\in\mathcal{H}^2_2(\mathbb{D})$ with power series representation $f(z)=\sum_{n=0}^{\infty}a_nz^n$. Then 
	\begin{align*}
		|f(z)|=\bigg|\sum_{n=0}^{\infty}a_nz^n\bigg|\le\sum_{n=0}^{\infty}|a_n||z|^n=\sum_{n=0}^{\infty}\sqrt{\frac{(n+2)^5}{n+1}}|a_n|\sqrt{\frac{n+1}{(n+2)^5}}|z|^n.
	\end{align*}
	By making use of the Cauchy-Schwarz inequality, we obtain that
	\begin{align*}
		|f(z)|\le&\sqrt{\bigg(\sum_{n=0}^{\infty}\frac{(n+2)^5}{n+1}|a_n|^2\bigg)}\sqrt{\bigg(\sum_{n=0}^{\infty}\frac{n+1}{(n+2)^5}|z|^{2n}\bigg)}\\=&\sqrt{\bigg(\sum_{n=0}^{\infty}\frac{n+1}{(n+2)^5}|z|^{2n}\bigg)}||f||_{\mathcal{H}^2_2}\\\le&\sqrt{\sum_{n=0}^{\infty}\frac{n+1}{(n+2)^5}}||f||_{\mathcal{H}^2_2}\\=&\sqrt{(\zeta(4)-\zeta(5))}||f||_{\mathcal{H}^2_2}.
	\end{align*}
	Hence, $||f||_{\infty}\le\sqrt{\zeta(4)-\zeta(5)}||f||_{\mathcal{H}^2_2}$.\vspace{1.2mm}
	
	If we consider the function  $f(z)=\sum_{n=0}^{\infty}((n+1)/(n+2)^5)z^n$, then $||f||_{\infty}=f(1)=\sum_{n=0}^{\infty}(n+1)/(n+2)^5=\zeta(4)-\zeta(5)$ and $||f||_{\mathcal{H}^2_2}=\sqrt{\zeta(4)-\zeta(5)}$. Hence, we conclude that the constant $\sqrt{\zeta(4)-\zeta(5)}$ is sharp. This completes the proof.
\end{proof}
Operator algebras constitute a core subject within functional analysis, a mathematical discipline concerned with vector spaces endowed with a topological structure. They offer a structure for comprehending the characteristics of linear operators, with a special focus on operators that are either bounded or unbounded. In 1969, the paper \cite{Kopp-PJM-1969} presented a proof that $D_{\alpha}$ is an algebra for $\alpha>1$. Furthermore, it was noted in \cite{Korenblum-MS-1972} that $\mathcal{S}^2_2(\mathbb{D})$ and its analogous counterparts $D_n$ for integer values of $n\ge1$ are also algebras, with reference to Russian literature. The research presented in \cite{Allen-Heller-Pons-ACM-2015} and \cite{Cuckovic-Paudyal-JMAA-2015} established that both $\mathcal{S}^2(\mathbb{D})$ and $\mathcal{S}^2_2(\mathbb{D})$ possess the characteristics of algebras. Recently, Guo and Luo \cite{Gu-Luo-CVEE-2018} demonstrated that $\mathcal{S}^2_1(\mathbb{D})$ is an algebra, and subsequently, Gupta and Malhotra \cite{Gupta-Malhotra-CVEE-2020} established that $\mathcal{H}^2_1(\mathbb{D})$ is also an algebra. Here, we will provide proof demonstrating that $\mathcal{H}^2_2(\mathbb{D})$ constitutes an algebra. 
\begin{prop}\label{pro-3.1}
	The space $\mathcal{H}^2_2(\mathbb{D})$ is an algebra.
\end{prop}
\begin{proof}[\bf Proof of Proposition \ref{pro-3.1}]
	Suppose $f\in\mathcal{H}^2_2(\mathbb{D})$ with power series representation $f(z)=\sum_{n=0}^{\infty}a_nz^n$. By simple computation, we obtain the following inequality
	\begin{align*}
		\bigg|f(z)-\sum_{n=0}^{N}a_nz^n\bigg|\le&\bigg|\sum_{n=N+1}^{\infty}\sqrt{\frac{(n+2)^5}{n+1}}a_n\sqrt{\frac{n+1}{(n+2)^5}}z^n\bigg|\\\le&\sum_{n=N+1}^{\infty}\sqrt{\frac{(n+2)^5}{n+1}}|a_n|\sqrt{\frac{n+1}{(n+2)^5}}|z|^n\\\le&\sqrt{\sum_{n=N+1}^{\infty}\frac{(n+2)^5}{n+1}|a_n|^2}\sqrt{\sum_{n=N+1}^{\infty}\frac{n+1}{(n+2)^5}|z|^{2n}}\\\le&||f||_{\mathcal{H}^2_2}\sqrt{\sum_{n=N+1}^{\infty}\frac{n+1}{(n+2)^5}}\\=&||f||_{\mathcal{H}^2_2}\sqrt{\sum_{n=N+1}^{\infty}\frac{1}{(n+2)^4}-\sum_{n=N+1}^{\infty}\frac{1}{(n+2)^5}}\\=&||f||_{\mathcal{H}^2_2}\sqrt{\sum_{n=0}^{\infty}\frac{1}{(n+N+3)^4}-\sum_{n=0}^{\infty}\frac{1}{(n+N+3)^5}}\\=&||f||_{\mathcal{H}^2_2}\sqrt{\frac{1}{6}Polygamma[3, N+3]-(-\frac{1}{24})Polygamma[4, N+3]}\\=&||f||_{\mathcal{H}^2_2}\sqrt{\frac{1}{24}(4Polygamma[3, N+3]+Polygamma[4, N+3])}
	\end{align*}
	(note that $Polygamma[n, z]$ represents the $n^{th}$ derivative, denoted as $\Psi^n(z)$, of the Digamma function $\Psi(z)$. The Digamma function $\Psi(z)$ serves as the logarithmic derivative of the gamma function $\Gamma(z)$, and we can express it as $\Psi(z) = (d/dz)\ln(\Gamma(z)) = \Gamma'(z)/\Gamma(z)$, where $\Gamma(z) = \int_{0}^{\infty}t^{t-1}e^{-t}dt$. The power series representation of the Polygamma function is as follows:
	\begin{align*}
		Polygamma[m, z] := \Psi^m(z) = (-1)^{m+1}m!\sum_{n=0}^{\infty}1/(n+z)^{m+1}.
	\end{align*}
	It should be noted that this representation is valid for values of $z$ greater than zero and for all natural numbers $n$.) By utilizing Mathematica software to observe that the limit of the expression 
	\begin{align*}
		\sqrt{\frac{1}{24}(4Polygamma[3, N+3]+Polygamma[4, N+3])}
	\end{align*} approaches zero as $N$ tends to $\infty$, we can conclude that $f(z)$ is the uniform limit of its partial sum. Consequently, $f(z)$ maintains a continuous extension to the boundary of $\mathbb{D}$. The space $\mathcal{H}^2_2(\mathbb{D})$ is a subset of the disk algebra  $A(\mathbb{D})$, where $A(\mathbb{D})$ comprises holomorphic functions mapping from $\mathbb{D}$ to $\mathbb{C}$ that exhibit continuous extensions to the boundary of $\mathbb{D}$. Let $A^1(\mathbb{D})=\{f\in A(\mathbb{D}):f^\prime(z)\in A(\mathbb{D})\}$. Then, $A^1(\mathbb{D})\subseteq\mathcal{H}^2_2(\mathbb{D})\subseteq A(\mathbb{D})$. Therefore, $\mathcal{H}^2_2(\mathbb{D})$ is an algebra. This completes the proof.
\end{proof}
In a recent investigation, Gu and Luo (see Theorem 2.2, \cite{Gu-Luo-CVEE-2018}) established that for $f$ and $g$ in $\mathcal{S}^2_1(\mathbb{D})$, the inequality 
\begin{align*}
	||fg||{\mathcal{S}^2_1}\le2\sqrt{2}||f||{\mathcal{S}^2_1}||g||{\mathcal{S}^2_1}
\end{align*} holds. Subsequently, Gupta and Malhotra (see Theorem 3.4, \cite{Gupta-Malhotra-CVEE-2020}) established that for $f$ and $g$ belonging to $\mathcal{H}^2_1(\mathbb{D})$, the inequality 
\begin{align*}
	||fg||^2_{\mathcal{H}^2_1}\le6(\pi^2/6-\zeta(3))||f||^2_{\mathcal{H}^2_1}||g||^2_{\mathcal{H}^2_1}
\end{align*} is valid. By a similar line of reasoning to that employed in Theorem 2.2 from \cite{Gu-Luo-CVEE-2018} and Theorem 3.4 from \cite{Gupta-Malhotra-CVEE-2020}, our objective in this section is to establish that, for $f$ and $g$ in $\mathcal{H}^2_2(\mathbb{D})$, the inequality \begin{align*}
||fg||_{\mathcal{H}^2_2}\le2\sqrt{2}\sqrt{\zeta(4)-\zeta(5)}||f||_{\mathcal{H}^2_2}||g||_{\mathcal{H}^2_2}
\end{align*} holds true.\vspace{1.2mm}

We obtain the following result for the class $\mathcal{H}^2_2(\mathbb{D})$ which is similar to \cite[Theorem 2.3]{Gu-Luo-CVEE-2018} and \cite[Theorem 3.5]{Gupta-Malhotra-CVEE-2020}.
\begin{thm}\label{thm-3.2}
	If $f, g\in\mathcal{H}^2_2(\mathbb{D})$, then $||fg||_{\mathcal{H}^2_2}\le2\sqrt{2}\sqrt{\zeta(4)-\zeta(5)}||f||_{\mathcal{H}^2_2}||g||_{\mathcal{H}^2_2}$. (note that $\zeta(4)-\zeta(5)\approxeq0.04539$).
\end{thm}
\begin{proof}[\bf Proof of Theorem \ref{thm-3.2}]
	Suppose that $f, g\in\mathcal{H}^2_2(\mathbb{D})$. Then  a straightforward computation leads to the following 
	\begin{align*}
		||fg||^2_{\mathcal{H}^2_2}=&31||fg||^2_{\mathcal{H}^2}+41||f^\prime g+fg^\prime||^2_{\mathcal{H}^2}+||(fg)^{\prime\prime}||^2_{\mathcal{H}^2}+||fg||^2_{\mathcal{A}^2}+49||f^\prime g+fg^\prime||^2_{\mathcal{A}^2}\\&+11||(fg)^{\prime\prime}||^2_{\mathcal{A}^2}\\\le&31(||f||^2_{\infty}||g||^2_{\mathcal{H}^2}+||f||^2_{\mathcal{H}^2}||g||^2_{\infty})+41(2||f||^2_{\infty}||g^{\prime}||^2_{\mathcal{H}^2}+2||f^{\prime}||^2_{\mathcal{H}^2}||g||^2_{\infty})\\&+(4||f||^2_{\infty}||g^{\prime\prime}||^2_{\mathcal{H}^2}+2||f^{\prime\prime}||^2_{\mathcal{H}^2}||g||^2_{\infty}+8||f||^2_{\infty}||g^{\prime}||^2_{\mathcal{H}^2}+8||f^{\prime}||^2_{\mathcal{H}^2}||g||^2_{\infty})\\&+(||f||^2_{\infty}||g||^2_{\mathcal{A}^2}+||f||^2_{\mathcal{A}^2}||g||^2_{\infty})+49(2||f||^2_{\infty}||g^{\prime}||^2_{\mathcal{A}^2}+2||f^{\prime}||^2_{\mathcal{A}^2}||g||^2_{\infty})\\&+11(4||f||^2_{\infty}||g^{\prime\prime}||^2_{\mathcal{H}^2}+2||f^{\prime\prime}||^2_{\mathcal{H}^2}||g||^2_{\infty}+8||f||^2_{\infty}||g^{\prime}||^2_{\mathcal{H}^2}+8||f^{\prime}||^2_{\mathcal{H}^2}||g||^2_{\infty})\\=&||f||^2_{\infty}(31||g||^2_{\mathcal{H}^2}+41||g^\prime||^2_{\mathcal{H}^2}+||g^{\prime\prime}||^2_{\mathcal{H}^2}+||g||^2_{\mathcal{A}^2}+49||g^\prime||^2_{\mathcal{A}^2}+11||g^{\prime\prime}||^2_{\mathcal{A}^2})\\&+||g||^2_{\infty}(31||f||^2_{\mathcal{H}^2}+41||f^{\prime}||^2_{\mathcal{H}^2}+||f^{\prime\prime}||^2_{\mathcal{H}^2}+||f||^2_{\mathcal{A}^2}+49||f^{\prime}||^2_{\mathcal{A}^2}+11||f^{\prime\prime}||^2_{\mathcal{A}^2})\\&+||f||^2_{\infty}(49||g^\prime||^2_{\mathcal{H}^2}+3||g^{\prime\prime}||^2_{\mathcal{H}^2}+137||g^\prime||^2_{\mathcal{A}^2}+33||g^{\prime\prime}||^2_{\mathcal{A}^2})+||g||^2_{\infty}(49||f^\prime||^2_{\mathcal{H}^2}\\&+||f^{\prime\prime}||^2_{\mathcal{H}^2}+137||f^\prime||^2_{\mathcal{A}^2}+11||f^{\prime\prime}||^2_{\mathcal{A}^2})\\\le&||f||^2_{\infty}||g||^2_{\mathcal{H}^2_2}+||g||^2_{\infty}||f||^2_{\mathcal{H}^2_2}+3||f||^2_{\infty}||g||^2_{\mathcal{H}^2_2}+3||g||^2_{\infty}||f||^2_{\mathcal{H}^2_2}\\=&4(||f||^2_{\infty}||g||^2_{\mathcal{H}^2_2}+||g||^2_{\mathcal{H}^2_2}||g||^2_{\infty})\\\le&8(\zeta(4)-\zeta(5))||f||^2_{\mathcal{H}^2_2}||g||^2_{\mathcal{H}^2_2}.
	\end{align*}
	Consequently, the desired inequality $||fg||_{\mathcal{H}^2_2}\le2\sqrt{2}\sqrt{(\zeta(4)-\zeta(5))}||f||_{\mathcal{H}^2_2}||g||_{\mathcal{H}^2_2}$ is established. This completes the proof. 
\end{proof}
\begin{thm}\label{thm-3.3}
	If $f\in H(\mathbb{D})$, then $M_f$ is a bounded linear operator on $\mathcal{H}^2_2(\mathbb{D})$ if and only if $f\in\mathcal{H}^2_2(\mathbb{D})$. Moreover, 
	\begin{align*}
		\max\{||f||_{\infty},||f||_{\mathcal{H}^2_2}\}\le||M_f||\le2\sqrt{2}\sqrt{(\zeta(4)-\zeta(5))}||f||_{\mathcal{H}^2_2}.
	\end{align*}
\end{thm}
\begin{proof}[\bf Proof of Theorem \ref{thm-3.3}]
	For all $w\in\mathbb{D}$, $||M^*_fK_w||=||\overline{f(w)}K_w||=|f(w)|||K_w||$. Therefore, $|f(w)|\le||M^*_f||=||M_f||$ for all $w\in\mathbb{D}$ and hence $||f||_{\infty}\le||M_f||$. Also, $||M_f||\ge||M_f1||=||f||_{\mathcal{H}^2_2}$. Consequently, $\max\{||f||_{\infty},||f||_{\mathcal{H}^2_2}\}\le||M_f||$. It follows from Theorem \ref{thm-3.2} that for all $g\in\mathcal{H}^2_2(\mathbb{D})$, we have 
	\begin{align*}
		||M_fg||=||fg||\le2\sqrt{2}\sqrt{(\zeta(4)-\zeta(5))}||f||_{\mathcal{H}^2_2}||g||_{\mathcal{H}^2_2}.
	\end{align*}
	By setting $g$ to be the identity function, we obtain that \begin{align*}
		||M_f||\le2\sqrt{2}\sqrt{(\zeta(4)-\zeta(5))}||f||_{\mathcal{H}^2_2}.
	\end{align*} 
	This leads to the following conclusion
	\begin{align*}
		\max\{||f||_{\infty},||f||_{\mathcal{H}^2_2}\}\le||M_f||\le2\sqrt{2}\sqrt{(\zeta(4)-\zeta(5))}||f||_{\mathcal{H}^2_2}.
	\end{align*}
	This completes the proof.
\end{proof}
\section{Complex symmetry of weighted composition operators on $\mathcal{H}^2_2(\mathbb{D})$}
Let $ \mathcal{B}(\mathcal{H}) $ be the algebra of all bounded linear operators on a separable complex Hilbert space $ \mathcal{H} $. A mapping $T:\mathcal{H}\to\mathcal{H}$ is said to be anti-linear if 
\begin{align*}
	T(ax+by)=\overline{a}T(x)+\overline{b}T(y),
\end{align*} 
for all $x, y\in\mathcal{H}$ and for all $a, b\in\mathbb{C}$.
A conjugation on $ \mathcal{H} $ is an antilinear operator $ \mathcal{C} : \mathcal{H} \to \mathcal{H}$ which satisfies $ \langle \mathcal{C}x, \mathcal{C}y \rangle_{\mathcal{H}}=\langle y, x \rangle_{\mathcal{H}} $ for all $ x, y\in\mathcal{H} $ and $ \mathcal{C}^2=I_{\mathcal{H}} $, where $ I_{\mathcal{H}} $ is the identity operator on $ {\mathcal{H}} $. An operator $ T\in\mathcal{B}({\mathcal{H}}) $ is said to be complex symmetric if there exists a conjugation $ \mathcal{C} $ on $ \mathcal{H} $ such that $ T=\mathcal{C}T^*\mathcal{C} $, or equivalently, $T\mathcal{C}=\mathcal{C}T^*$. In this case, we say that $ T $ is complex symmetric with conjugation $ \mathcal{C} $.
\begin{defn}
An operator $T$ is said to be a Hilbert-Schmidt operator in a Hilbert space $\mathcal{H}$ if there exists an orthonormal basis $\{e_n\}_{n\ge0}$ for which the series $\sum_{n=0}^{\infty}||Te_n||^2$ converges to a finite value.
\end{defn}
 In \cite{Gupta-Malhotra-CVEE-2020}, Gupta and Malhotra have established that if $\phi$ belongs to the space $\mathcal{H}^2_1(\mathbb{D})$ and satisfies $||\phi||_{\infty}<1$, the operator $C_{\phi}$ is a Hilbert-Schmidt operator. In the next result, it is shown that the composition operator $C_{\phi}$ on $\mathcal{H}^2_2(\mathbb{D})$ is a Hilbert-Schmidt operator, and it is compact when $\phi$ belongs to $\mathcal{H}^2_2(\mathbb{D})$ and satisfies the condition $||\phi||_{\infty}<1$. The proof for the following result follows a similar line of reasoning as that of \cite[Theorem 4.2]{Gupta-Malhotra-CVEE-2020}.

\begin{thm}\label{thm-4.1}
	Let $\phi\in\mathcal{H}^2_2(\mathbb{D})$ be such that $||\phi||_{\infty}<1$. Then the composition operator $C_{\phi}$ is a Hilbert-Schmidt operator on $\mathcal{H}^2_2(\mathbb{D})$.
\end{thm}
\begin{proof}[\bf Proof of Theorem \ref{thm-4.1}]
	In order to establish that the composition operator $C_{\phi}$ is a Hilbert-Schmidt operator, it is necessary to demonstrate that the series $\sum_{n=0}^{\infty}||C_{\phi}e_n(z)||^2$ has a finite sum. An easy computation shows that 
	
	\begin{align*}
		\sum_{n=0}^{\infty}||C_{\phi}e_n(z)||^2=\sum_{n=0}^{\infty}\bigg|\bigg|C_{\phi}\bigg(\frac{z^n}{||z^n||}\bigg)\bigg|\bigg|^2=1+\sum_{n=1}^{\infty}\frac{||\phi^n||^2_{\mathcal{H}^2_2}}{||z^n||^2_{\mathcal{H}^2_2}}=1+\sum_{n=1}^{\infty}\frac{(n+1)||\phi^n||^2_{\mathcal{H}^2_2}}{(n+2)^5}.
	\end{align*}
	and 
	\begin{align*}
		||\phi^n||^2_{\mathcal{H}^2_2}=&31||\phi\phi^{n-1}||^2_{\mathcal{H}^2}+41||(\phi^n)^{\prime}||^2_{\mathcal{H}^2}+||(\phi^n)^{\prime\prime}||^2_{\mathcal{H}^2}+||\phi\phi^{n-1}||^2_{\mathcal{A}^2}+49||(\phi^n)^{\prime}||^2_{\mathcal{A}^2}\\&+11||(\phi^n)^{\prime\prime}||^2_{\mathcal{A}^2}\\=&31||\phi\phi^{n-1}||^2_{\mathcal{H}^2}+41||n\phi^{n-1}\phi^{\prime}||^2_{\mathcal{H}^2}+||n(n-1)\phi^{n-2}(\phi^{\prime})^2+n\phi^{n-1}\phi^{\prime\prime}||^2_{\mathcal{H}^2}\\&+||\phi\phi^{n-1}||^2_{\mathcal{A}^2}+49||n\phi^{n-1}\phi^{\prime}||^2_{\mathcal{A}^2}+11||n(n-1)\phi^{n-2}(\phi^{\prime})^2+n\phi^{n-1}\phi^{\prime\prime}||^2_{\mathcal{A}^2}\\\le&31||\phi^{n-1}||^2_{\infty}||\phi||^2_{\mathcal{H}^2}+41n^2||\phi^{n-1}||^2_{\infty}||\phi^{\prime}||^2_{\mathcal{H}^2}+2[n^2(n-1)^2||\phi^{n-1}||^2_{\infty}||\phi^{\prime}||^2_{\mathcal{H}^2}\\&+n^2||\phi^{n-1}||^2_{\infty}||\phi^{\prime\prime}||^2_{\mathcal{H}^2}]+||\phi^{n-1}||^2_{\infty}||\phi||^2_{\mathcal{A}^2}+49n^2||\phi^{n-1}||^2_{\infty}||\phi^{\prime}||^2_{\mathcal{A}^2}\\&+22[n^2(n-1)^2||\phi^{n-1}||^2_{\infty}||\phi^{\prime}||^2_{\mathcal{A}^2}+n^2||\phi^{n-1}||^2_{\infty}||\phi^{\prime\prime}||^2_{\mathcal{A}^2}]\\\le&||\phi^{n-1}||^2_{\infty}[31||\phi||^2_{\mathcal{H}^2}+41n^2||\phi^{\prime}||^2_{\mathcal{H}^2}+n^2||\phi^{\prime\prime}||^2_{\mathcal{H}^2}+||\phi||^2_{\mathcal{A}^2}+49n^2||\phi^{\prime}||^2_{\mathcal{A}^2}\\&+11n^2||\phi^{\prime\prime}||^2_{\mathcal{A}^2}]+||\phi^{n-1}||^2_{\infty}[2n^2(n-1)^2\phi^{\prime}||^2_{\mathcal{H}^2}+n^2||\phi^{\prime\prime}||^2_{\mathcal{H}^2}\\&+22n^2(n-1)^2||\phi^{\prime}||^2_{\mathcal{A}^2}+11n^2||\phi^{\prime\prime}||^2_{\mathcal{A}^2}]\\\le&n^2||\phi^{n-1}||^2_{\infty}[31||\phi||^2_{\mathcal{H}^2}+41||\phi^{\prime}||^2_{\mathcal{H}^2}+||\phi^{\prime\prime}||^2_{\mathcal{H}^2}+||\phi||^2_{\mathcal{A}^2}+49||\phi^{\prime}||^2_{\mathcal{A}^2}+11||\phi^{\prime\prime}||^2_{\mathcal{A}^2}]\\&+n^4||\phi^{n-1}||^2_{\infty}[2\phi^{\prime}||^2_{\mathcal{H}^2}+||\phi^{\prime\prime}||^2_{\mathcal{H}^2}+22||\phi^{\prime}||^2_{\mathcal{A}^2}+11||\phi^{\prime\prime}||^2_{\mathcal{A}^2}]\\\le&2n^4||\phi||^{2(n-1)}_{\infty}||\phi||^2_{\mathcal{H}^2_2}.
	\end{align*}
	Thus, it follows that 
	\begin{align*}
		\sum_{n=0}^{\infty}\bigg|\bigg|C_{\phi}\bigg(\frac{z^n}{||z^n||}\bigg)\bigg|\bigg|^2&\le1+\sum_{n=1}^{\infty}\frac{2n^4(n+1)||\phi||^{2(n-1)}_{\infty}||\phi||^2_{\mathcal{H}^2_2}}{(n+2)^5}\\&\le1+\sum_{n=1}^{\infty}\frac{2n^4(n+1)||\phi||^{2(n-1)}_{\infty}||\phi||^2_{\mathcal{H}^2_2}}{(n+2)^5}\\&\le1+2||\phi||^2_{\mathcal{H}^2_2}\sum_{n=1}^{\infty}||\phi||^{2(n-1)}_{\infty}\\&=1+2||\phi||^2_{\mathcal{H}^2_2}(1-||\phi||^2_{\infty})^{-1}<\infty.
	\end{align*}
	Hence, $C_{\phi}$ is a Hilbert-Schmidt operator. 
\end{proof}
Our next result is for complex symmetric with conjugation $\mathcal{J}$ of the composition operator $C_{\phi}$. 
\begin{thm}\label{thm-4.2}
	Let $\phi$ be a self holomorphic map on $\mathbb{D}$ such that the composition operator $C_{\phi}:\mathcal{H}^2_2(\mathbb{D})\to\mathcal{H}^2_2(\mathbb{D})$ is complex symmetric with conjugation $\mathcal{J}$. Then $\phi(z)=az$, where $a=\phi^{\prime}(0), \phi(0)=0$. Conversely, if $\phi(z)=az$ maps the unit disk into itself where $a\in\mathbb{D}$, then the composition operator $C_{\phi}:\mathcal{H}^2_2(\mathbb{D})\to\mathcal{H}^2_2(\mathbb{D})$ is complex symmetric with conjugation $\mathcal{J}$. 
\end{thm}
\begin{proof}[\bf Proof of Theorem \ref{thm-4.2}]
	Suppose $C_{\phi}$ is complex symmetric with conjugation $\mathcal{J}$. Since the linear span of reproducing kernels is dense in any reproducing kernel Hilbert space, for all $w,z\in\mathbb{D}$, we have 
	\begin{align*}
		(C_{\phi}\mathcal{J})K_w(z)=(\mathcal{J}C^*_{\phi})k_w(z),
	\end{align*} 
	that is, $K_{\overline{w}(\phi(z))}=K_{\overline{\phi(w)}(z)}$. Thus, for all $w,z\in\mathbb{D}$,
	\begin{align}\label{Eq-4.1}
		\sum_{n=0}^{\infty}\frac{n+1}{(n+2)^5}(\phi(z)w)^n=	\sum_{n=0}^{\infty}\frac{n+1}{(n+2)^5}(z\phi(w))^n.
	\end{align}
	Taking the derivative with respect to $z$ on the both sides in \eqref{Eq-4.1}, we obtain
	\begin{align}\label{Eq-4.2}
		\sum_{n=0}^{\infty}\frac{n(n+1)}{(n+2)^5}(\phi(z)^{n-1}\phi^{\prime}(z)w^n=	\sum_{n=0}^{\infty}\frac{n(n+1)}{(n+2)^5}z^{n-1}\phi(w)^n.	
	\end{align}
	On substituting $z=0$ in \eqref{Eq-4.2}, we obtain
	\begin{align}\label{Eq-4.3}
		\phi^{\prime}(0)\sum_{n=0}^{\infty}\frac{n(n+1)}{(n+2)^5}(\phi(0)^{n-1}w^n=\frac{2}{243}\phi(w)\;\; \mbox{for all}\;\; w\in\mathbb{D}.
	\end{align}
	Again putting $w=0$ in \eqref{Eq-4.3}, it follows that $\phi(0)=0$. Hence, from \eqref{Eq-4.3} we get $\phi(z)=az$ where $a=\phi^{\prime}(0)$.\vspace{1.2mm}
	
	Conversely, suppose $\phi(z)=az$ is a self holomorphic map on $\mathbb{D}$ where $a\in\mathbb{D}$. It can be easily verified that for all $w,z\in\mathbb{D}$
	\begin{align*}
		K_{\overline{w}(\phi(z))}=\sum_{n=0}^{\infty}\frac{n+1}{(n+2)^5}=K_{\overline{\phi(w)}}(z).
	\end{align*}
	This proves that $C_{\phi}$ is complex symmetric with conjugation $\mathcal{J}$.
\end{proof}
We obtain the following result showing that if $ W_{\Psi, \phi} $ is complex symmetric with conjugation $\mathcal{J}$, then explicit form of $ \phi(z) $ and $ \Psi(z) $ can be obtained.
\begin{thm}\label{thm-3.4}
	Let $\Psi:\mathbb{D}\to\mathbb{C}$ and $\phi:\mathbb{D}\to\mathbb{D}$ be holomorphic maps such that the weighted composition operator $W_{\Psi, \phi}:\mathcal{H}^2_2(\mathbb{D})\to\mathcal{H}^2_2(\mathbb{D})$ is complex symmetric with conjugation $\mathcal{J}$. Then for all $z\in\mathbb{D}$,
	\begin{align*}
		\phi(z)=a_0+a_1\frac{q(z)}{p(z)}, \; \Psi(z)=a_2p(z),
	\end{align*}
	where $a_0=\phi(0), a_1=\phi^{\prime}(0), a_2=3888\Psi(0),$
	\[
	\begin{cases}
		p(z)=\displaystyle\frac{2}{243}\sum_{n=0}^{\infty}((n+1)/(n+2)^5)a^n_0z^n\\ q(z)=\displaystyle\frac{1}{32}\sum_{n=1}^{\infty}(n(n+1)/(n+2)^5)a^{n-1}_0z^n.
	\end{cases}
	\]Conversely, let $a_0,a_1\in\mathbb{D}$ and $a_2\in\mathbb{C}$. If $\phi(z)=a_0+a_1(q(z)/p(z))$ and $\Psi(z)=a_2p(z)$ are holomorphic maps on $\mathbb{D}$, where 
	
	\begin{align}\label{eq-3.1}
		\begin{cases}
			p(z)=\displaystyle\frac{2}{243}\sum_{n=0}^{\infty}((n+1)/(n+2)^5)a^n_0z^n;\\
			q(z)=\displaystyle\frac{1}{32}\sum_{n=1}^{\infty}(n(n+1)/(n+2)^5)a^{n-1}_0z^n
		\end{cases}
	\end{align} 
	then the weighted composition operator $W_{\Psi, \phi}:\mathcal{H}^2_2(\mathbb{D})\to\mathcal{H}^2_2(\mathbb{D})$ is complex symmetric with conjugation $\mathcal{J}$ only if $a_0=0$ or $a_1=0$ or both are zero.
\end{thm}
\begin{proof}[\bf Proof of Theprem \ref{thm-3.4}]
	First we suppose that $W_{\Psi, \phi}$ is complex symmetric with conjugation $\mathcal{J}$. Then $(W_{\Psi, \phi}\mathcal{J})K_w(z)=(\mathcal{J}W^*_{\Psi, \phi})K_w(z)$ for all $w, z\in\mathbb{D}$ and it follows that $\Psi(z)K_{\overline{w}}(\phi(z))=\Psi(w)K_{\overline{\phi(w)}}(z)$. This implies that 
	\begin{align}
		\Psi(z)\sum_{n=0}^{\infty}\frac{n+1}{(n+2)^5}(\phi(z)w)^n=\Psi(w)\sum_{n=0}^{\infty}\frac{n+1}{n+2)^5}(z\phi(w))^n.
	\end{align}
	Suppose that
	\begin{align*}
		g_z(w)=\sum_{n=0}^{\infty}\frac{n+1}{(n+2)^5}(\phi(z)w)^n\;\; \mbox{and}\;\; h_z(w)=\sum_{n=0}^{\infty}\frac{n+1}{(n+2)^5}(z\phi(w))^n
	\end{align*}
	for all $w, z\in\mathbb{D}$. Hence, it follows that 
	\begin{align}\label{Eq-3.3}
		\Psi(z)g_z(w)=\Psi(w)h_z(w)\;\; \mbox{for all}\;\; w, z\in\mathbb{D}.
	\end{align}
	Setting $w=0$ in \eqref{Eq-3.3}, we obtain that $\psi(z)g_z(0)=\Psi(0)h_z(w)$ and hence 
	\begin{align}\label{Eq-3.4}
		\Psi(z)=\Psi(0)\frac{h_z(0)}{g_z(0)}\;\; \mbox{for all}\;\; w, z\in\mathbb{D}.
	\end{align}
	By using \eqref{Eq-3.4}, we obtain from \eqref{Eq-3.3} that 
	\begin{align}\label{Eq-3.5}
		\Psi(0)\frac{h_z(0)}{g_z(0)}g_z(w)=\Psi(0)\frac{h_w(0)}{g_w(0)}h_z(w).
	\end{align}
	It can be easily shown that $g_z(w)=h_w(z)$ for all $w,z \in\mathbb{D}$. Using this in \eqref{Eq-3.5}, we have $(h_z(0)/g_z(0))g_z(w)=(g_0(w)/h_0(w))h_z(w)$ for all $w, z\in\mathbb{D}$. Therefore, we see that
	\begin{align}\label{Eq-3.6}
		\bigg(\frac{h_z(0)}{g_z(0)}\bigg)g_z(w)h_0(w)=g_0(w)h_z(w).
	\end{align}
	Taking the derivative of \eqref{Eq-3.6} with respect to $w$, we obtain
	\begin{align}\label{Eq-3.7}
		\bigg(\frac{h_z(0)}{g_z(0)}\bigg)(g^{\prime}_z(w)h_0(w)+g_z(w)h^{\prime}_0(w))=g^{\prime}_0(w)h_z(w)+g_0(w)h^{\prime}_z(w).
	\end{align}
	Putting $w=0$ in \eqref{Eq-3.7}, we obtain 
	\begin{align}\label{Eq-3.8}
		h_z(0)(g^{\prime}_z(0)h_0(0)+g_z(0)h^{\prime}_0(0))=g_z(0)(g^{\prime}_0(0)h_z(0)+g_0(0)h^{\prime}_z(0))
	\end{align}
	By using the defining formula of $g_z(w)$ and $h_z(w)$, we obtain
	\begin{align*}
		g_z(0)&=g_0(0)=h_0(0)=\frac{1}{32};\\g^{\prime}_z(0)&=\frac{2}{243}\phi(z);\\g^{\prime}_0(0)&=\frac{2}{243}\phi(0);\\h_z(0)&=\sum_{n=0}^{\infty}\frac{n+1}{(n+2)^5}(z\phi(0))^n;\\h^{\prime}_z(0)&=\sum_{n=1}^{\infty}\frac{n(n+1)}{(n+2)^5}(\phi(0))^{n-1}\phi^{\prime}(0)z^n;\\h^{\prime}_0(0)&=0.
	\end{align*}
	Denoting $a_0 = \phi(0)$ and $a_1 = \phi^{\prime}(0)$, and then substituting the above expressions into equation \eqref{Eq-3.8}, we obtain
	\begin{align*}
		\bigg(\sum_{n=0}^{\infty}\frac{n+1}{(n+2)^5}a^n_0z^n\bigg)\bigg(\frac{2}{243}\phi(z)\bigg)&=\frac{2}{243}a_0\sum_{n=0}^{\infty}\frac{n+1}{(n+2)^5}a^n_0z^n+\frac{1}{32}a_1\sum_{n=1}^{\infty}\frac{n(n+1)}{(n+2)^5}a^{n-1}_0z^n.\\\Rightarrow \phi(z)&=a_0+a_1\frac{q(z)}{p(z)},
	\end{align*}
	where 
	\begin{align*}
		p(z)=\frac{2}{243}\sum_{n=0}^{\infty}\frac{n+1}{(n+2)^5}a^n_0z^n \;\;\mbox{and}\;\; q(z)=\frac{1}{32}\sum_{n=1}^{\infty}\frac{n(n+1)}{(n+2)^5}a^{n-1}_0z^n.
	\end{align*}
	Take $a_2 = 3888\Psi(0)$. Consequently, it follows that $\phi(z) = a_0 + a_1(q(z)/p(z))$ and $\Psi(z) = a_2p(z)$ for all $z\in\mathbb{D}$.\vspace{1.2mm}
	
	Conversely, suppose that $\phi(z) = a_0 + a_1(q(z)/p(z))$ and $\Psi(z) = a_2p(z)$ where $a_0, a_1\in\mathbb{D}, a_2\in\mathbb{C}$ and $p(z)$ and $q(z)$ are as in \eqref{eq-3.1}. Let $w$ and $z$ be elements of $\mathbb{D}$. For $W_{\Psi, \phi}$ to have complex symmetry, the following must hold:
	\begin{align}
		\Psi(z)K_{\overline{w}}(\phi(z))=\Psi(w)K_{\overline{\phi(w)}}(z),
	\end{align}
	that is 
	\begin{align}\label{Eq-3.10}
		\bigg(&\sum_{n=0}^{\infty}\frac{n+1}{(n+2)^5}a^n_0z^n\bigg)\bigg(\sum_{n=0}^{\infty}\frac{n+1}{(n+2)^5}\bigg(a_0w+a_1\frac{q(z)}{p(z)}w\bigg)^n\bigg)\\&=\bigg(\sum_{n=0}^{\infty}\frac{n+1}{(n+2)^5}a^n_0w^n\bigg)\bigg(\sum_{n=0}^{\infty}\frac{n+1}{(n+2)^5}\bigg(a_0z+a_1\frac{q(w)}{p(w)}z\bigg)^n\bigg)\nonumber.
	\end{align}
	In the entirety of the unit disk $\mathbb{D}$, the function $q(z)/p(z)$ is holomorphic, and it satisfies the condition $q(0)=0$. Consequently, we can express $q(z)/p(z)$ as an infinite series:
	\begin{align*}
		\frac{q(z)}{p(z)}=\sum_{i=1}^{\infty}c_ia^{i-1}_0z^i,
	\end{align*}
	where $c_1=1$, and for $i=2, 3, \ldots$, the coefficients $c_i$ are positive real numbers less than one. Thus, \eqref{Eq-3.10} is equivalent to 
	\begin{align*}
		\bigg(&\sum_{n=0}^{\infty}\frac{n+1}{(n+2)^5}a^n_0z^n\bigg)\bigg(\sum_{m=0}^{\infty}\frac{m+1}{(m+2)^5}\bigg(\sum_{k=0}^{m}\binom{m}{k}(a_0w)^k\bigg(\sum_{i=1}^{\infty}a_1c_ia^{i-1}_0z^iw\bigg)^{m-k}\bigg)\bigg)\\&=\bigg(\sum_{n=0}^{\infty}\frac{n+1}{(n+2)^5}a^n_0w^n\bigg)\\&\times\bigg(\sum_{m=0}^{\infty}\frac{m+1}{(m+2)^5}\bigg(\sum_{k=0}^{m}\binom{m}{k}(a_0z)^k\bigg(\sum_{i=1}^{\infty}a_1c_ia^{i-1}_0w^iz\bigg)^{m-k}\bigg)\bigg).
	\end{align*}
	On comparing the coefficients of $z^3w$, on L.H.S, we obtain
	\begin{align}\label{Eq-3.11}
		\frac{2}{3^5}\bigg(\frac{4}{5^5}a^4_0+\frac{3}{4^5}a^2_0a_1c_1+\frac{2}{3^5}a^2_0a_1c_2+\frac{1}{2^5}a^2_0a_1c_3\bigg),
	\end{align}
	whereas on R.H.S, we obtain
	\begin{align}\label{Eq-3.12}
		\frac{4}{5^5}\bigg(\frac{2}{3^5}a^4_0+\frac{3}{2^5}a^2_0a_1c_1\bigg).
	\end{align}
	Consequently, we find that the expressions of \eqref{Eq-3.11} and \eqref{Eq-3.12} will be equal only if $a_0$ is zero or $a_1$ is zero or both are zero. With this the proof is completed.
\end{proof}

\vspace{2mm}
\noindent\textbf{Compliance of Ethical Standards}\\

\noindent\textbf{Conflict of interest} The authors declare that there is no conflict  of interest regarding the publication of this paper.\vspace{1.5mm}

\noindent\textbf{Data availability statement}  Data sharing not applicable to this article as no datasets were generated or analysed during the current study.

\end{document}